\documentclass[reqno,12pt]{amsart}
\def\vint_#1{\mathchoice%
          {\mathop{\kern 0.2em\vrule width 0.6em height 0.69678ex depth -0.58065ex
                  \kern -0.8em \intop}\nolimits_{\kern -0.4em#1}}%
          {\mathop{\kern 0.1em\vrule width 0.5em height 0.69678ex depth -0.60387ex
                  \kern -0.6em \intop}\nolimits_{#1}}%
          {\mathop{\kern 0.1em\vrule width 0.5em height 0.69678ex
              depth -0.60387ex
                  \kern -0.6em \intop}\nolimits_{#1}}%
          {\mathop{\kern 0.1em\vrule width 0.5em height 0.69678ex depth -0.60387ex
                  \kern -0.6em \intop}\nolimits_{#1}}}
\def\vintslides_#1{\mathchoice%
          {\mathop{\kern 0.1em\vrule width 0.5em height 0.697ex depth -0.581ex
                  \kern -0.6em \intop}\nolimits_{\kern -0.4em#1}}%
          {\mathop{\kern 0.1em\vrule width 0.3em height 0.697ex depth -0.604ex
                  \kern -0.4em \intop}\nolimits_{#1}}%
          {\mathop{\kern 0.1em\vrule width 0.3em height 0.697ex de pth -0.604ex
                  \kern -0.4em \intop}\nolimits_{#1}}%
          {\mathop{\kern 0.1em\vrule width 0.3em height 0.697ex depth -0.604ex
                  \kern -0.4em \intop}\nolimits_{#1}}}

\DeclareFontFamily{OMX}{yhex}{}
\DeclareFontShape{OMX}{yhex}{m}{n}{<->yhcmex10}{}
\DeclareSymbolFont{yhlargesymbols}{OMX}{yhex}{m}{n}
\DeclareMathAccent{\wideparen}{\mathord}{yhlargesymbols}{"F3}

\usepackage[T1]{fontenc}

\usepackage{a4wide}
\usepackage{mathtools}
\usepackage{amsthm}
\usepackage{amssymb}
\usepackage[colorlinks,citecolor=green,linkcolor=red]{hyperref}
\usepackage{color}
\mathtoolsset{showonlyrefs}
\usepackage[obeyFinal]{todonotes}
\usepackage{csquotes}
\usepackage{graphicx}
\usepackage{psfrag}
\usepackage{epsfig,color}
\usepackage{enumitem}

\numberwithin{equation}{section}
\newtheorem{theorem}{Theorem}[section]
\newtheorem{lemma}[theorem]{Lemma}

\theoremstyle{definition}
\newtheorem{definition}[theorem]{Definition}
\newtheorem{proposition}[theorem]{Proposition}
\newtheorem{example}[theorem]{Example}
\theoremstyle{remark}
\newtheorem{rem}[theorem]{Remark}

\newcommand{\R}{\mathbb{R}}

\newcommand{\N}{\mathbb{N}}

\renewcommand{\H}{\mathcal{H}}

\newcommand{\diam}{\mathrm{diam}}
\newcommand{\conv}{\mathrm{conv}}
\newcommand{\dist}{\mathrm{dist}}

\begin{document}

\title[Sharp estimate on the inner distance in planar domains]
{Sharp estimate on the inner distance\\ in planar domains}

\author{Danka Lu\v{c}i\'c} 
\author{Enrico Pasqualetto}
\author{Tapio Rajala}

\address{University of Jyvaskyla\\
         Department of Mathematics and Statistics \\
         P.O. Box 35 (MaD) \\
         FI-40014 University of Jyvaskyla \\
         Finland}
\email{danka.d.lucic@jyu.fi}
\email{enrico.e.pasqualetto@jyu.fi}
\email{tapio.m.rajala@jyu.fi}

\thanks{All authors partially supported by the Academy of Finland, projects 274372, 307333, 312488, and 314789.}
\date{\today}


\begin{abstract}
 We show that the inner distance inside a bounded planar domain is at most the one-dimensional Hausdorff measure of the boundary of the domain. We prove this sharp result by establishing an improved Painlev\'e length estimate for connected sets and by using the metric removability of totally disconnected sets, proven by Kalmykov, Kovalev, and Rajala. We also give a totally disconnected example showing that for general sets the Painlev\'e length bound $\kappa(E) \le\pi \H^1(E)$ is sharp.
\end{abstract}

\maketitle
%
%
\section{Introduction}
In this paper we continue the study of the internal distance for planar domains. For a domain $\Omega\subset \R^2$, the internal distance $d_\Omega \colon \Omega^2 \to [0,\infty)$ is defined as
\[
d_\Omega(x,y) := \inf\left\{\ell(\gamma)\,:\,\gamma \text{ is a curve connecting }x \text{ to }y\right\},
\]
where $\ell(\gamma)$ denotes the length of the curve $\gamma$.
The internal distance is determined by how much the boundary blocks the curves $\gamma$. One result in this direction was proven in \cite{KKR}: If the complement of the domain $\Omega$ is totally disconnected with finite $\H^1$-measure, then $d_\Omega$ is the Euclidean distance. In other words, totally disconnected closed sets with finite $\H^1$-measure are (metrically) removable.
The proof of this result used the estimate
\begin{equation}\label{eq:known_inner_dist}
d_\Omega(x,y)\leq|x-y|+\frac{\pi}{2}\,\H^1(\partial\Omega).
\end{equation}
We improve \eqref{eq:known_inner_dist} to the following sharp estimate:
\begin{theorem}\label{thm:main}
Let $\Omega \subset \R^2$ be a domain satisfying $\H^1(\partial\Omega)<\infty$. Then the estimate
\begin{equation}\label{eq:inner_dist_unbdd}
d_\Omega(x,y) \le |x-y| + \H^1(E)
\end{equation}
holds for every $x,y\in \Omega$,
where $E \subset \partial\Omega$ is the union of all the connected components of $\partial\Omega$ with positive length.
In the case when $\Omega$ is bounded, the above estimate can be improved to
\begin{equation}\label{eq:inner_dist_bdd}
d_\Omega(x,y) \le \H^1(E).
\end{equation}
\end{theorem}
The assumption $\H^1(\partial\Omega)< \infty$ in Theorem \ref{thm:main} is used for showing that the totally disconnected part of the boundary is removable. In view of the examples constructed in \cite{HakobyanHerron}, it is at least necessary to assume that the Hausdorff dimension of $\partial\Omega$ is at most one. However, it is not clear if the assumption $\H^1(\partial\Omega)< \infty$ could be relaxed to $\partial\Omega$ having $\sigma$-finite $\H^1$-measure.

The sharpness of the estimate \eqref{eq:inner_dist_unbdd} in the unbounded case is seen simply by taking $\partial \Omega$ to be a line-segment. In the bounded case, the sharpness is seen for example by considering
\[
\Omega = (0,1)^2\setminus \bigcup_{i=1}^n (\{\frac1{2i}\}\times[0,1-1/i]) \cup (\{\frac1{2i+1}\}\times[1/i,1])
\]
for $n$ larger and larger, and by scaling $\Omega$.

As a consequence of Theorem \ref{thm:main}, we obtain the following result:
\begin{theorem}\label{thm:accessibility_thm}
Let \(\Omega\subset\R^2\) be a bounded domain
with \(\H^1(\partial\Omega)<\infty\). Let \(x\in\Omega\)
and \(y\in\partial\Omega\) be given. Then for every \(\varepsilon>0\)
there exists an injective Lipschitz curve \(\gamma\colon[0,1]\to\R^2\)
joining \(x\) to \(y\) such that \(\gamma|_{(0,1)}\subset\Omega\)
and \(\ell(\gamma)\leq\H^1(\partial\Omega)+\varepsilon\).
\end{theorem}
The previous result can be proven by arguing as in the proof of
Theorem \ref{thm:accessible}, but replacing the estimate
\eqref{eq:known_inner_dist} with \eqref{eq:inner_dist_bdd}.
\begin{rem}
We point out that the curve \(\gamma\) in Theorem \ref{thm:accessibility_thm} can be 
chosen to be smooth in the open interval \((0,1)\), as follows from a standard 
approximation argument.
\end{rem}
The paper is organized as follows. In Section \ref{sec:aux} we recall, and prove, basic results in planar geometry; especially for planar domains whose boundary has finite length. In Section \ref{sec:Painleve} we show an improved version of the Painlev\'e length estimate for connected sets and show the sharpness of the general Painlev\'e length estimate for disconnected sets. In the final Section \ref{sec:proof} we prove our main theorem, Theorem \ref{thm:main}.
%
%
%
%
%
\section{Some auxiliary results}\label{sec:aux}
We collect in this section some standard results in planar geometry
that will be needed in the remaining part of this paper.
An open, connected subset of \(\R^2\) is referred to as a \emph{(planar) domain}.
By \emph{Jordan loop} we mean a closed simple curve \(\sigma\colon[0,1]\to\R^2\).
%
%
\begin{lemma}\label{lem:diam_vs_H1}
Let \(C\subset\R^2\) be a connected set. Then it holds that
\[
\H^1(C)\geq|x-y|\quad\text{ for every }x,y\in C.
\]
In particular, we have that \(\diam(C)\leq\H^1(C)\).
\end{lemma}
\begin{proof} Fix \(x\in C\) and consider the function
\(f_x\colon \R^2\to \R\)
defined by \(f_x(y)\coloneqq |y-x|\) for every \(y\in \R^2\). Observe that 
the function \(f_x\) is \(1\)-Lipschitz, so that 
\(\H^1(C)\geq \mathcal L^1\big(f_x(C)\big)\). Moreover, since \(C\) is connected, 
we know that \(\big[0,|y-x|\big]=\big[f_x(x),f_x(y)\big]\subset f_x(C)\)
holds for every \(y\in C\). Consequently, we have that 
\(|y-x|\leq \mathcal L^1\big(f_x(C)\big)\) for every \(y\in C\).
Therefore we conclude that 
\[
\H^1(C)\geq \mathcal L^1\big(f_x(C)\big)\geq |x-y|\quad \text{ for every }
x,y\in C,
\]
as required. Taking the supremum over \(x,y\in C\)
we also get that \({\rm diam}(C)\leq \H^1(C)\).
\end{proof}
For a proof of the following fact we refer, e.g., to \cite[Fact 3.1]{HakobyanHerron}:
\begin{lemma}\label{lem:fact_HH}
Let \(\Omega\) be a domain in \(\R^2\). Let \(F\) be some connected
component of \(\partial\Omega\). Denote by \(B\) the connected component
of \(\Omega^c\) that contains \(F\). Then \(\partial B=F\).
\end{lemma}
%
%
%
%
A domain \(\Omega\subset \R^2\) is said to be \emph{locally connected along
its boundary} provided for every point \(x\in\partial\Omega\) and every
radius \(r>0\) there exists \(t\in(0,r)\) such that \(\Omega\cap B_t(x)\)
is contained in one connected component of \(\Omega\cap B_r(x)\).
\medskip

The following result has been stated and proved in \cite[Corollary 3.3]{HakobyanHerron}:
\begin{proposition}\label{prop:HH2}
Let \(\Omega\subset\R^2\) be any domain such that \(\R^2\setminus\Omega\)
is connected and not a singleton. Suppose that \(\Omega\) is locally connected
along its boundary and \(\partial\Omega\) is bounded. Then \(\partial\Omega\)
is a Jordan loop.
\end{proposition}
As an immediate consequence, we can obtain the following result:
\begin{theorem}\label{thm:H1_finite_Jordan_domain}
Let \(\Omega\) be a bounded domain in \(\R^2\) with \(\H^1(\partial\Omega)<+\infty\).
Let \(U\) be a connected component of \(\R^2\setminus\overline\Omega\).
Then \(\partial U\) is a Jordan loop (of finite length). Moreover, it holds that:
\begin{itemize}
\item[\(\rm i)\)] If \(U\) is bounded, then \(\Omega\) lies in the unbounded
connected component of \(\R^2\setminus\partial U\).
\item[\(\rm ii)\)] If \(U\) is unbounded, then \(\Omega\) lies in the bounded
connected component of \(\R^2\setminus\partial U\).
\end{itemize}
\end{theorem}
\begin{proof}
Consider a connected component \(U\) of \(\R^2\setminus\overline\Omega\)
and denote $V\coloneqq\R^2\setminus\overline U$.
In view of Proposition \ref{prop:HH2}, it is sufficient to show that
\(U\) is locally connected along its boundary. Fix any \(x\in\partial U\)
and \(r>0\). We claim that: 
\begin{equation}\label{eq:enclosing_loop_claim}
\text{Only finitely many connected components of }U\cap B_{2r}(x)
\text{ intersect }B_r(x).
\end{equation}
Call \(\mathcal F\) such family. Suppose by contradiction that \(\mathcal F\)
is not a finite set and call \(\{E_i\}_{i\in I}\) its elements. Notice that
for any \(i\in I\) we have that \(E_i\cap\partial B_{2r}(x)\neq\emptyset\)
(as \(U\) is connected), thus also \(E_i\cap\partial B_r(x)\neq\emptyset\)
(as \(E_i\) intersects \(B_r(x)\) by definition of \(\mathcal F\)).
In particular, one has \(E_i\cap\partial B_\lambda(x)\neq\emptyset\) for all
\(\lambda\in(r,2r)\). Furthermore, it holds that
\(V\cap\partial B_\lambda(x)\neq\emptyset\) for every \(\lambda\in(r,2r)\)
(otherwise there would be just one element in \(\mathcal F\)). This implies that
\begin{equation}\label{eq:enclosing_loop_aux}
\partial E_i\cap\partial B_\lambda(x)\neq\emptyset
\quad\text{ for every }i\in I\text{ and }\lambda\in(r,2r).
\end{equation}
Let us define \(\Gamma_i\coloneqq\partial E_i\cap B_{2r}(x)\) for every \(i\in I\).
Then \(\{\Gamma_i\}_{i\in I}\) are pairwise disjoint Borel subsets of 
\(\partial U\), see Figure \ref{fig:longbound}.
We infer from the property \eqref{eq:enclosing_loop_aux}
that \(\H^1(\Gamma_i)\geq r\) for every \(i\in I\), whence accordingly it holds
that \(\H^1(\partial\Omega)\geq\H^1(\partial U)\geq\sum_{i\in I}\H^1(\Gamma_i)=+\infty\).
This leads to a contradiction, thus proving the claim \eqref{eq:enclosing_loop_claim}.
\begin{figure}
    \centering
    \includegraphics[width=0.4\columnwidth]{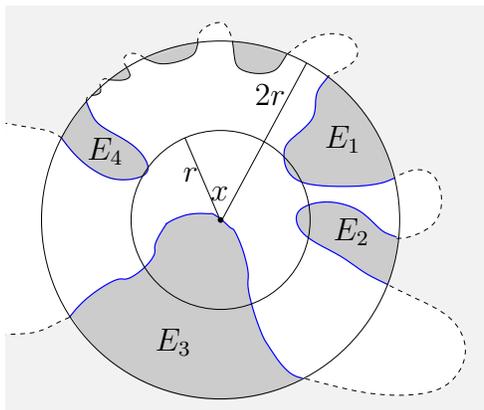}
    \caption{Proofs of Theorem \ref{thm:H1_finite_Jordan_domain} and Lemma \ref{lem:accessible_pts_aux} rely on the finiteness of the length of the boundary of $\Omega$ to deduce that there are only finitely many components intersecting both the small ball and the complement of the large ball.}
    \label{fig:longbound}
\end{figure}

Since \(\mathcal F\) has finite cardinality, we can find \(t\in(0,r)\) so
small that the element of \(\mathcal F\) containing \(x\) is the only one
that intersects \(B_t(x)\). This forces \(U\cap B_t(x)\) to be contained in
one connected component of \(U\cap B_{2r}(x)\). Therefore \(U\) is connected
along its boundary, as required. The proof of items i) and ii) follows by a
standard topological argument.
\end{proof}
\begin{lemma}\label{lem:accessible_pts_aux}
Let \(\Omega\subset\R^2\) be a bounded domain
satisfying \(\H^1(\partial\Omega)<+\infty\). Fix any \(x\in\partial\Omega\)
and \(\varepsilon>0\). Then there exists a subdomain \(\Omega'\) of \(\Omega\)
with \(\H^1(\partial\Omega')\leq\varepsilon\) such that \(x\in\partial\Omega'\)
and \(\diam(\Omega')\leq\varepsilon\). Moreover, we can further require
that \(\partial\Omega'\subset\partial\Omega\cup\partial B(x,r)\)
for some \(r>0\).
\end{lemma}
\begin{proof} {\color{blue}\textsc{Step 1.}}
Given that \(\H^1|_{\partial\Omega}\) is upper continuous,
we can choose \(r\in\big(0,\varepsilon/(4\,\pi)\big)\) such that
\(\H^1\big(\partial\Omega\cap B(x,r)\big)\leq\varepsilon/2\).
If \(\Omega\subset\overline B(x,r)\), then the set \(\Omega'\coloneqq\Omega\)
does the job, thus let us assume that there exists a point
\(y\in\Omega\setminus\overline B(x,r)\). Let us denote by \(\mathcal F\) the
family of all connected components of \(\Omega\cap B(x,r)\) that
intersect \(\partial B(x,r/2)\). Given any \(V\in\mathcal F\), it holds:
\begin{equation}\label{eq:access_0}\begin{split}
&\text{There is a continuous curve }\alpha\colon[0,1]\to\R^2\text{ such that}\\
&\alpha|_{(0,1]}\subset V\text{, }\alpha_0\in\partial B(x,r)\cap\Omega
\text{ and }\alpha_1\in\partial B(x,r/2).
\end{split}\end{equation}
Indeed, we can find a continuous curve \(\alpha'\colon[0,1]\to\Omega\)
joining \(y\) to a point of \(\partial B(x,r/2)\cap V\). Calling
\(t_0\coloneqq\max\big\{t\in[0,1]\,\big|\,\alpha'_t\in\partial B(x,r)\big\}<1\),
we see that the curve \(\alpha\) obtained by restricting \(\alpha'\)
to \([t_0,1]\) fulfills the requirements.\\
{\color{blue}\textsc{Step 2.}} Moreover, we claim that:
\begin{equation}\label{eq:access_1}
\sharp(\partial V_1\cap\partial V_2\cap\partial V_3)\leq 2
\quad\text{ if }V_1,V_2,V_3\in\mathcal F\text{ are distinct.}
\end{equation}
Call \(\sigma_1,\sigma_2,\sigma_3\) those Jordan loops such that
\(\sigma_i\subset\partial V_i\) and \(V_i\) is contained in the bounded
connected component \(U_i\) of \(\R^2\setminus\sigma_i\) for \(i=1,2,3\)
(cf.\ Theorem \ref{thm:H1_finite_Jordan_domain}).
It follows from \eqref{eq:access_0} that the sets \(U_1,U_2,U_3\) are
pairwise disjoint. Suppose by contradiction that there exist at least
three distinct points \(z_1,z_2,z_3\) in
\(\partial V_1\cap\partial V_2\cap\partial V_3\). In particular, such points
are forced to belong to \(\sigma_1\cap\sigma_2\cap\sigma_3\). Fix some other
points \(a_j\in U_j\) for \(j=1,2,3\). We can build continuous curves
\(\gamma^{ij}\colon[0,1]\to\R^2\) for \(i,j=1,2,3\) such that the following
properties hold:
\begin{itemize}
\item[\(\rm i)\)] \(\gamma^{ij}\) joins \(z_i\) to \(a_j\) for all \(i,j=1,2,3\).
\item[\(\rm ii)\)] \(\gamma^{ij}|_{(0,1)}\subset U_j\) for all \(i,j=1,2,3\).
\item[\(\rm iii)\)] \(\gamma^{1j}|_{(0,1)},\gamma^{2j}|_{(0,1)},\gamma^{3j}|_{(0,1)}\)
are pairwise disjoint for all \(j=1,2,3\).
\end{itemize}
This would imply that the complete 3-by-3 bipartite graph \(K_{3,3}\) is planar,
thus leading to a contradiction (cf.\ \cite[Chapter I, Theorem 17]{Bollobas}).
Therefore, the claim \eqref{eq:access_1} is proven.\\
{\color{blue}\textsc{Step 3.}} As a consequence, we can show that:
\begin{equation}\label{eq:access_2}
\mathcal F\text{ is a finite family.}
\end{equation}
In order to prove it, we argue by contradiction: suppose \(\mathcal F\)
is infinite, say \(\mathcal F=(V_i)_{i\in\N}\). Thanks to \eqref{eq:access_0}
we know that \(V_i\cap\partial B(x,\lambda)\neq\emptyset\) for any \(i\in\N\)
and \(\lambda\in(r/2,r)\). Since \(\mathcal F\) contains more than one
element, we infer that also \(\partial V_i\cap\partial B(x,\lambda)\neq\emptyset\)
holds for any \(i\in\N\) and \(\lambda\in(r/2,r)\). In particular, we have
that \(\H^1(\partial V_i)\geq r/2\) for every \(i\in\N\). On the other hand,
\eqref{eq:access_1} grants that 
\(\H^1\big(\bigcup_{i\in\N}\partial V_i\big)\geq 2\sum_{i\in\N}\H^1(\partial V_i)\).
This implies that \(\bigcup_{i\in\N}\partial V_i\) has infinite \(\H^1\)-measure,
which is in contradiction with the fact that
\(\bigcup_{i\in\N}\partial V_i\subset\partial\Omega\cup\partial B(x,r)\).
Accordingly, property \eqref{eq:access_2} is verified.\\
{\color{blue}\textsc{Step 4.}} We also claim that:
\begin{equation}\label{eq:access_3}
\Omega\cap B(x,r/2)\subset U\coloneqq\bigcup_{V\in\mathcal F}V.
\end{equation}
Indeed, fix any \(z\in B(x,r/2)\cap\Omega\). Choose a continuous curve
\(\alpha\colon[0,1]\to\Omega\) such that \(\alpha_0=z\) and \(\alpha_1=y\).
Call \(t_0\coloneqq\min\big\{t\in[0,1]\,\big|\,\alpha_t\in\partial B(x,r)\big\}>0\).
Then \(\alpha|_{[0,t_0)}\) is a connected subset of \(\Omega\cap B_r(x)\) that
intersects \(\partial B(x,r/2)\), whence it is contained in some element of
\(\mathcal F\). This shows that \(z\in\bigcup_{V\in\mathcal F}V\), which
yields the claim \eqref{eq:access_3}.\\
{\color{blue}\textsc{Step 5.}} We can finally conclude the proof by combining
\eqref{eq:access_2} with \eqref{eq:access_3}: the latter ensures that
\(x\in\partial U\), thus the former implies that \(x\in\partial\Omega'\)
for some element \(\Omega'\in\mathcal F\). Given that we have
\(\Omega'\subset B(x,r)\subset B(x,\varepsilon/2)\) (so that
\(\diam(\Omega')\leq\varepsilon\)) and
\[
\H^1(\partial\Omega')\leq\H^1\big(\partial\Omega\cap B(x,r)\big)
+\H^1\big(\partial B(x,r)\big)\leq\frac{\varepsilon}{2}+2\,\pi\,r\leq\varepsilon,
\]
the statement is achieved.
\end{proof}
\begin{theorem}[Accessible points]\label{thm:accessible}
Let \(\Omega\subset\R^2\) be a bounded domain
with \(\H^1(\partial\Omega)<\infty\). Let \(x\in\Omega\)
and \(y\in\partial\Omega\) be given. Then for every \(\varepsilon>0\)
there exists an injective Lipschitz curve \(\gamma\colon[0,1]\to\R^2\)
joining \(x\) to \(y\) such that \(\gamma|_{(0,1)}\subset\Omega\)
and \(\ell(\gamma)\leq|x-y|+\frac{\pi}{2}\,\H^1(\partial\Omega)+\varepsilon\).
\end{theorem}
\begin{proof} {\color{blue}\textsc{Step 1.}}
Fix \(\varepsilon\in(0,1)\). Call \(x_0\coloneqq x\) and \(\Omega_0\coloneqq\Omega\).
By applying Lemma \ref{lem:accessible_pts_aux} in a recursive way, we
build a decreasing sequence \((\Omega_n)_{n\geq 1}\) of subdomains of
\(\Omega\setminus\{x_0\}\) satisfying the following properties:
\begin{itemize}
\item[\(\rm i)\)] \(\H^1(\partial\Omega_n)<\varepsilon/(2^{n+1}\,\pi)\) for all \(n\geq 1\).
\item[\(\rm ii)\)] \(y\in\partial\Omega_n\) for all \(n\geq 1\).
\item[\(\rm iii)\)] \(\diam(\Omega_n)\leq\varepsilon/2^{n+2}\) for all \(n\geq 1\).
\item[\(\rm iv)\)] \(\dist(\partial\Omega_n\cap\Omega,\partial\Omega_{n+1}\cap\Omega)>0\)
for all \(n\geq 0\).
\end{itemize}
See Figure \ref{fig:access} for an illustration of the sequence of subdomains.
\begin{figure}
    \centering
    \includegraphics[width=0.5\columnwidth]{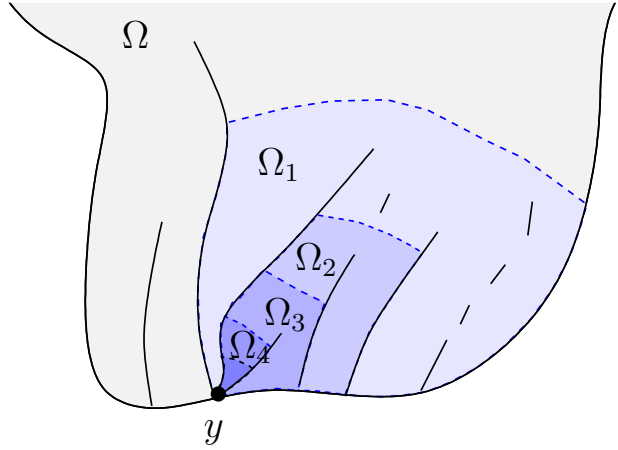}
    \caption{In the proof of Theorem \ref{thm:accessible} we use Lemma  \ref{lem:accessible_pts_aux} to find a nested sequence of subdomains. The next step is to then connect points in subsequent subdomains by using the non-optimal inner distance estimate \eqref{eq:known_inner_dist}.}
    \label{fig:access}
\end{figure}
For brevity, let us set \(I_n\coloneqq[1-2^{-n},1-2^{-n-1}]\) for all \(n\in\N\).
We claim that we can build a sequence \((x_n)_{n\geq 1}\subset\Omega\) such that
\(x_n\in\Omega_n\setminus\overline\Omega_{n+1}\) for all \(n\geq 1\), and
a sequence of injective Lipschitz curves \((\alpha^n)_{n\geq 0}\) such that
each \(\alpha^n\colon I_n\to\Omega_n\setminus\overline\Omega_{n+2}\) joins
\(x_n\) to \(x_{n+1}\) and satisfies
\begin{equation}\label{eq:length_alpha_n}\begin{split}
\ell(\alpha^0)&\leq|x-y|+\frac{\pi}{2}\,\H^1(\partial\Omega)+\frac{\varepsilon}{2},\\
\ell(\alpha^n)&\leq\frac{\varepsilon}{2^{n+1}}\quad\text{ for every }n\geq 1.
\end{split}\end{equation}
First, fix any sequence of points \((x'_n)_{n\geq 1}\subset\Omega\) that satisfies
\(x'_n\in\Omega_n\) for all \(n\geq 1\). 
By using \eqref{eq:known_inner_dist}, we can find an injective curve
\(\tilde\alpha^0\colon[0,1]\to\Omega_0\) joining \(x_0\) to \(x'_1\) such that
\(\ell(\tilde\alpha^0)<|x_0-x'_1|+\frac{\pi}{2}\,\H^1(\partial\Omega)+\varepsilon/4\).
Given that \(|x_0-x'_1|\leq|x-y|+|y-x'_1|\leq|x-y|+\varepsilon/4\) is verified
by items ii) and iii) above, we see that
\(\ell(\tilde\alpha^0)\leq|x-y|+\frac{\pi}{2}\,\H^1(\partial\Omega)+\frac{\varepsilon}{2}\).
Item iv) tells us that there exists \(t_1\in(0,1]\) such that
\(\tilde\alpha^0|_{[0,t_1]}\) lies in \(\Omega_0\setminus\overline\Omega_2\)
and \(x_1\coloneqq\tilde\alpha^0_{t_1}\in\Omega_1\). Therefore, the injective
Lipschitz curve \(\alpha^0\colon I_0\to\Omega_0\) obtained by reparametrizing
\(\tilde\alpha^0|_{[0,t_1]}\) satisfies the first line of \eqref{eq:length_alpha_n}.
Now suppose to have already defined \(x_0,\ldots,x_n\)
and \(\alpha^0,\ldots,\alpha^{n-1}\) with the required properties. 
We can use again property \eqref{eq:known_inner_dist}
and item i) to find an injective curve \(\tilde\alpha^n\colon[0,1]\to\Omega_n\)
joining \(x_n\) to \(x'_{n+1}\) such that
\(\ell(\tilde\alpha^n)<|x_n-x'_{n+1}|+\varepsilon/2^{n+2}\).
Since the points \(x_n,x'_{n+1}\) are in \(\Omega_n\), we infer from item iii) that
\(|x_n-x'_{n+1}|\leq\varepsilon/2^{n+2}\) and accordingly we have
\(\ell(\tilde\alpha^n)\leq\varepsilon/2^{n+1}\). We can choose
\(t_{n+1}\in(0,1]\) so that \(\tilde\alpha^n|_{[0,t_{n+1}]}\) lies in
\(\Omega_n\setminus\overline\Omega_{n+2}\) and
\(x_{n+1}\coloneqq\tilde\alpha^n_{t_{n+1}}\in\Omega_{n+1}\).
Therefore, the injective Lipschitz curve \(\alpha^n\colon I_n\to\Omega_n\)
obtained by reparametrizing \(\tilde\alpha^n|_{[0,t_{n+1}]}\) satisfies the
second line of \eqref{eq:length_alpha_n}. This proves our claim.

Consider the unique continuous curve \(\alpha\colon[0,1)\to\Omega\)
satisfying \(\alpha|_{I_n}\coloneqq\alpha^n\) for all \(n\in\N\).
Items ii) and iii) grant that \(|\alpha_t-y|\leq 2^{-n}\) holds for all \(n\in\N\)
and \(t\in\bigcup_{k\geq n}I_k\). This ensures that \(\lim_{t\nearrow 1}\alpha_t=y\)
whence \(\alpha\) can be extended to a continuous curve \(\alpha\colon[0,1]\to\R^2\)
joining \(x\) to \(y\). By using \eqref{eq:length_alpha_n} we also deduce that
\begin{equation}\label{eq:length_alpha}
\ell(\alpha)\leq|x-y|+\frac{\pi}{2}\,\H^1(\partial\Omega)+\frac{\varepsilon}{2}+
\sum_{n=1}^\infty\frac{\varepsilon}{2^{n+1}}
=|x-y|+\frac{\pi}{2}\,\H^1(\partial\Omega)+\varepsilon.
\end{equation}
{\color{blue}\textsc{Step 2.}} Observe that the curve \(\alpha\) might not be
injective. We thus proceed as follows: we recursively build a sequence
\((\gamma^n)_{n\geq 1}\) of curves defined on \([0,1]\) such that
\begin{itemize}
\item[\(\rm a)\)] \(\gamma^n\) is a constant-speed, \(\ell(\alpha)\)-Lipschitz
and injective curve for all \(n\geq 1\),
\item[\(\rm b)\)] \(\gamma^n\) joins \(x\) to \(x_n\) for all \(n\geq 1\),
\item[\(\rm c)\)] the image of \(\gamma^n\) lies in \(\alpha^0\cup\ldots\cup\alpha^{n-1}\)
for all \(n\geq 1\),
\item[\(\rm d)\)] calling
\(s_n\coloneqq\min\big\{t\in[0,1]\,\big|\,\gamma^n_t\in\partial\Omega_n\big\}\)
it holds \(\gamma^n|_{[0,s_n]}\subset\gamma^{n+1}\) for all \(n\geq 1\).
\end{itemize}
First, we take as \(\gamma^1\colon[0,1]\to\Omega\) the constant-speed
reparametrization of \(\alpha^0\). Now suppose to have already defined
\(\gamma^n\) for some \(n\geq 1\). Let us define
\(t'\coloneqq\max\{t\in I_n\,:\,\alpha^n_t\in\gamma^n\}\) and choose that
\(t''\in[0,1]\) for which \(\gamma^n_{t''}=\alpha^n_{t'}\). Therefore, we
call \(\gamma^{n+1}\colon[0,1]\to\Omega\) the constant-speed reparametrization
of the concatenation between \(\gamma^n|_{[0,t'']}\) and
\(\alpha^n|_{[t',+\infty)\cap I_n}\). It follows from the very construction
that \(\gamma^{n+1}\) satisfies items a), b), c) and d), as required.

The Ascoli-Arzel\`{a} theorem grants that (possibly passing to a not relabeled
subsequence) the curves \(\gamma^n\) uniformly converge to some limit curve
\(\gamma\colon[0,1]\to\R^2\) with \(\gamma_0=x\). By using item a) and the lower
semicontinuity of the length functional, we deduce that \(\ell(\gamma)\leq\ell(\alpha)\).
By item b) we know that \(\gamma_1=\lim_n\gamma^n_1=\lim_n x_n=y\).
Given \(n\geq 1\), we set \(S_n\coloneqq\gamma^n\big([0,s_n]\big)\)
and \(\lambda_n\coloneqq\H^1(S_n)\). Item d) ensures that \(S_n\subset\gamma^k\)
for all \(k\geq n\), thus \(\ell(\gamma^k)\geq\lambda_n\). Thanks to item c)
we also see that \(\gamma^k\subset\gamma\cup\alpha^{k-1}\cup\alpha^k\), so that
\[
\ell(\gamma^k)\leq\ell(\gamma)+\ell(\alpha^{k-1})+\ell(\alpha^k)
\leq\ell(\gamma)+\frac{\varepsilon}{2^k}+\frac{\varepsilon}{2^{k+1}}
\leq\ell(\gamma)+\frac{1}{2^{n-1}}\eqqcolon q_n.
\]
Given that \(\gamma^k\big(\big[0,\lambda_n/\ell(\gamma^k)\big]\big)\subset S_n\)
and \(\lambda_n/\ell(\gamma^k)\geq\lambda_n/q_n\), we have
\(\gamma^k\big([0,\lambda_n/q_n]\big)\subset S_n\). Observe also that
\begin{equation}\label{eq:estimate_d_S_n}
d_{S_n}(\gamma^k_t,\gamma^k_s)=\ell(\gamma^k)\,|t-s|\geq\lambda_n\,|t-s|
\quad\text{ for every }k\geq n\text{ and }t,s\in[0,\lambda_n/q_n].
\end{equation}
Since \(\gamma^n|_{[0,s_n]}\colon[0,s_n]\to S_n\) is a homeomorphism,
we conclude from \eqref{eq:estimate_d_S_n} by letting \(k\to\infty\)
that \(d_{S_n}(\gamma_t,\gamma_s)\geq\lambda_n\,|t-s|\) for every
\(t,s\in[0,\lambda_n/q_n]\). In particular, one has that
\begin{equation}\label{eq:gamma_inj}
\gamma\text{ is injective on }[0,\lambda_n/q_n]\text{ for every }n\geq 1.
\end{equation}
Notice that \(\alpha^0\cup\ldots\cup\alpha^{n-2}\subset S_n\) for all \(n\geq 2\)
by construction, whence \(\gamma^n\subset S_n\cup\alpha^{n-1}\) by item c) and
accordingly \(\ell(\gamma^n)\leq\lambda_n+\ell(\alpha^{n-1})\). This implies that
\[
1\geq\varliminf_{n\to\infty}\frac{\lambda_n}{q_n}=
\frac{\varliminf\nolimits_n\lambda_n}{\lim\nolimits_n q_n}\geq
\frac{\varliminf\nolimits_n\ell(\gamma^n)-\lim\nolimits_n\ell(\alpha^{n-1})}{\ell(\gamma)}
\geq 1.
\]
Therefore, we deduce that \([0,1)=\bigcup_{n\geq 1}[0,\lambda_n/q_n]\), so that
\eqref{eq:gamma_inj} grants that \(\gamma\) is injective on \([0,1)\).
Since \(\gamma\big([0,\lambda_n/q_n]\big)\subset S_n\subset\alpha\setminus\{y\}\)
for all \(n\geq 1\), we see that
\(\gamma|_{[0,1)}\subset\alpha\setminus\{y\}\subset\Omega\). Finally, the curve
\(\gamma\) is \(\ell(\alpha)\)-Lipschitz and
\(\ell(\alpha)\leq|x-y|+\frac{\pi}{2}\,\H^1(\partial\Omega)+\varepsilon\)
by \eqref{eq:length_alpha}, whence
\(\ell(\gamma)\leq|x-y|+\frac{\pi}{2}\,\H^1(\partial\Omega)+\varepsilon\) as well.
This completes the proof of the statement.
\end{proof}
\begin{lemma}\label{lem:touch_boundary}
Let \(K\subset\R^2\) be a compact, connected set. Let \(\Omega\subset\R^2\) be
an open set such that \(K\setminus\bar\Omega\neq\emptyset\). Then for any
connected component \(E\) of \(K\cap\bar\Omega\) it holds that
\(E\cap\partial\Omega\neq\emptyset\).
\end{lemma}
\begin{proof}
It readily follows, e.g., from \cite[Lemma 2.14]{AlbOtt17}.
\end{proof}
%
%
%
For the sake of completeness, we report a proof of the ensuing standard fact:
\begin{proposition}\label{prop:boundary_convex_hull}
Let \(K\subset\R^2\) be a compact connected set such that \(\H^1(K)<\infty\).
Let us denote by \(C\) the convex hull of \(K\). Then
\begin{equation}\label{eq:boundary_convex_hull}
\H^1(\partial C)\leq 2\,\H^1(K).
\end{equation}
\end{proposition}
\begin{proof}
We subdivide the proof into several steps:\\
{\color{blue}\textsc{Step 1.}} If \(\mathring C=\emptyset\) then \(C\) is a segment,
thus \(\partial C=C=K\) and accordingly \eqref{eq:boundary_convex_hull} is
trivially verified. Now assume that \(\mathring C\neq\emptyset\), so that
(by convexity of \(C\)) we know that there exists a continuous curve
\(\gamma\colon[0,1]\to\R^2\), which is injective on \([0,1)\), such that
\(\gamma_0=\gamma_1\in K\) and \(\gamma[0,1]=\partial C\).
We can write \(\gamma^{-1}(\R^2\setminus K)=\bigcup_{i\in I}(a_i,b_i)\),
where \(I\subset\N\) and the sets \((a_i,b_i)\) are pairwise disjoint
subintervals of \((0,1)\). Since \(C\) is the convex hull of \(K\),
we have that the set \(S_i\coloneqq\gamma(a_i,b_i)\subset\partial C\setminus K\)
is a segment for all \(i\in I\).\\
{\color{blue}\textsc{Step 2.}} Given any \(i\in I\), let us call \(m_i\) the midpoint
of \(S_i\), while \(v_i\in\R^2\) stands for the unit vector perpendicular to
\(S_i\) such that \(m_i+\R^{>0}v_i\) intersects \(C\). Since \(K\)
is compact, there exists \(\varepsilon_i>0\) such that
\((m_i+[0,\varepsilon_i]v_i)\cap K=\emptyset\). Then we denote by \(A_i\)
the connected component of \(\R^2\setminus(\partial C\cup K)\) containing
\(m_i+\varepsilon_i v_i\). Observe that \(A_i\) is open
(as \(\partial C\cup K\) is compact) and that \(S_i\subset\partial A_i\).
Moreover, we claim that
\begin{equation}\label{eq:boundary_convex_hull_claim1}
A_i\cap A_j=\emptyset\quad\text{ for every }i,j\in I\text{ with }i\neq j.
\end{equation}
We argue by contradiction: suppose that \(A_i\cap A_j\neq\emptyset\), thus
necessarily \(A_i=A_j\). By Theorem \ref{thm:accessible} we know that there
exists an injective continuous curve \(\sigma\colon[0,1]\to\R^2\) such
that \(\sigma_0=m_i\), \(\sigma_1=m_j\) and \(\sigma(0,1)\subset A_i\).
Possibly interchanging \(i\) and \(j\), we can assume that \(b_i<a_j\).
Choose \(t\in(b_i,a_j)\) such that \(\gamma_t\in K\). Since we are supposing
that the set of indexes \(I\) has cardinality at least \(2\), we know that
\(K\cap\mathring C\neq\emptyset\) (otherwise \(K\) would be contained
in \(\partial C\) and accordingly disconnected). Pick any \(y\in K\cap\mathring C\).
Either \(\gamma_0\) or \(\gamma_t\) does not belong to the closure of the
connected component of \(\R^2\setminus(\partial C\cup\sigma[0,1])\) containing \(y\).
Since \(K\) is connected, this forces \(K\) to intersect \(\sigma(0,1)\),
which leads to a contradiction. Therefore the claim
\eqref{eq:boundary_convex_hull_claim1} is proven.\\
{\color{blue}\textsc{Step 3.}} Given any \(i\in I\), we call \(R_i\) the
strip \(S_i+\R^{>0}v_i\). We also define
\(B_i\coloneqq R_i\cap\partial A_i\cap\partial C\) and
\(C_i\coloneqq(R_i\cap\partial A_i)\setminus \partial C\), which are disjoint
Borel subsets of \(K\) by \eqref{eq:boundary_convex_hull_claim1}. We claim that
\begin{subequations}\begin{align}
\H^0(B_i\cap B_j)\leq 2&\quad\text{ for every }i,j\in I\text{ with }i\neq j,
\label{eq:boundary_convex_hull_claim2}\\
\H^0(C_i\cap C_j\cap C_k)\leq 1&\quad\text{ for every }i,j,k\in I
\text{ with }i\neq j\neq k\neq i.
\label{eq:boundary_convex_hull_claim3}
\end{align}\end{subequations}
We argue by contradiction. In order to prove \eqref{eq:boundary_convex_hull_claim2},
suppose to have three distinct points \(y_1,y_2,y_3\) in \(B_i\cap B_j\).
The set \(\partial C\setminus\{y_1,y_2,y_3\}\) is made of three arcs.
By Theorem \ref{thm:accessible} we can find injective continuous curves
\(\sigma,\sigma'\colon[0,1]\to\R^2\) such that \(\sigma_0=m_i\),
\(\sigma_1=y_2\), \(\sigma(0,1)\subset A_i\), \(\sigma'_0=m_j\),
\(\sigma'_1=y_1\) and \(\sigma'(0,1)\subset A_j\). We now distinguish two cases:
\begin{itemize}
\item[\(\rm i)\)] \(S_i\) and \(S_j\) lie in the same arc of
\(\partial C\setminus\{y_1,y_2,y_3\}\). Then (up to relabeling \(y_1,y_2,y_3\))
it holds that \(\sigma[0,1]\cap\sigma'[0,1]\neq\emptyset\),
thus contradicting \eqref{eq:boundary_convex_hull_claim1}.
\item[\(\rm ii)\)] \(S_i\) and \(S_j\) lie in different arcs of
\(\partial C\setminus\{y_1,y_2,y_3\}\). Possibly relabeling \(y_1,y_2,y_3\),
we have that \(S_i\) is between \(y_1\) and \(y_3\),
while \(S_j\) is between \(y_2\) and \(y_3\).
Then it holds that \(\sigma[0,1]\cap\sigma'[0,1]\neq\emptyset\), again
contradicting \eqref{eq:boundary_convex_hull_claim1}.
\end{itemize}
Therefore \eqref{eq:boundary_convex_hull_claim2} is proven.
In order to prove \eqref{eq:boundary_convex_hull_claim3}, suppose that
\(C_i\cap C_j\cap C_k\) contains at least two distinct points \(z_1,z_2\).
By using Theorem \ref{thm:accessible} we can build an injective continuous curve
\(\sigma\colon[0,2]\to\R^2\) with \(\sigma_0=m_i\), \(\sigma_1=z_1\),
\(\sigma_2=m_j\), \(\sigma(0,1)\subset A_i\) and \(\sigma(1,2)\subset A_j\).
Possibly interchanging \(z_1\) and \(z_2\), we can assume that \(z_2\) and
\(m_k\) do not belong to the same connected component of \(C\setminus\sigma(0,2)\).
Hence (again by Theorem \ref{thm:accessible}) we can pick an injective continuous
curve \(\sigma'\colon[0,1]\to\R^2\) such that \(\sigma'_0=m_k\),
\(\sigma'_1=z_2\) and \(\sigma'(0,1)\subset A_k\). This implies that
\(\sigma(0,2)\cap\sigma'(0,1)\neq\emptyset\), whence either
\(A_i\cap A_k\neq\emptyset\) or \(A_j\cap A_k\neq\emptyset\).
In both cases property \eqref{eq:boundary_convex_hull_claim1} is violated,
thus even the claim \eqref{eq:boundary_convex_hull_claim3} is proven.\\
{\color{blue}\textsc{Step 4.}} Let \(i\in I\) be fixed. For any \(x\in S_i\)
there is \(t_x>0\) such that \(p_x\coloneqq x+t_x v_i\in B_i\cup C_i\)
and \(\big(x+(0,t_x)v_i\big)\cap\partial A_i=\emptyset\). Call
\(\pi_i\colon\R^2\to\R^2\) the orthogonal projection onto the line
containing \(S_i\), which is a \(1\)-Lipschitz map. Then one has that
\[\H^1(S_i)=\H^1\big(\pi_i\{p_x\,:\,x\in S_i\}\big)
\leq\H^1\big(\{p_x\,:\,x\in S_i\}\big)\leq\H^1(B_i)+\H^1(C_i).\]
Therefore it holds that
\begin{equation}\label{eq:boundary_convex_hull_claim4}
\H^1(\partial C)=\H^1(\partial C\cap K)+\sum_{i\in I}\H^1(S_i)
\leq\H^1(\partial C\cap K)+\sum_{i\in I}\H^1(B_i)+\sum_{i\in I}\H^1(C_i).
\end{equation}
Finally, we easily deduce from \eqref{eq:boundary_convex_hull_claim2} and
\eqref{eq:boundary_convex_hull_claim3} that
\(\sum_{i\in I}\H^1(B_i)\leq\H^1(\partial C\cap K)\) and
\(\sum_{i\in I}\H^1(C_i)\leq 2\,\H^1(\mathring C\cap K)\), respectively.
By plugging these estimates into \eqref{eq:boundary_convex_hull_claim4},
we conclude that \eqref{eq:boundary_convex_hull} is satisfied.
This completes the proof of the statement.
\end{proof}
%
%
%
%
\begin{proposition}\label{prop:move_points}
Let \(\Gamma\subset\R^2\) be a Borel set satisfying \(\H^1(\Gamma)<+\infty\).
Let \(v\in\R^2\setminus\{0\}\) and \(x\in\R^2\) be given. Fix a point \(y\in\R^2\)
that does not belong to the line \(x+\R v\). Then
\[
\H^1\big([x+tv,y]\cap\Gamma\big)=0\quad\text{ for a.e.\ }t\in\R.
\]
\end{proposition}
\begin{proof}
Notice that the elements of \(\big\{[x+tv,y)\cap\Gamma\big\}_{t\in\R}\) are pairwise
disjoint subsets of \(\Gamma\). Given that \(\H^1|_\Gamma\) is a finite Borel measure
on \(\R^2\), we conclude that \(\H^1\big([x+tv,y)\cap\Gamma\big)=0\)
for all but countably many \(t\in\R\), whence the statement follows.
\end{proof}
\begin{rem}\label{rmk:conn_inside}
Let \(\Omega\) be an open, connected subset of \(\R^2\).
Let \(K\) be a compact subset of \(\Omega\). Then there
exists an open, connected set \(U\subset\R^2\) such that
\(K\subset U\) and \(\overline U\subset\Omega\).

Indeed, the compactness of \(K\) ensures that we can find a finite family
\(B_1,\ldots,B_n\) of open balls such that
\(K\subset\bigcup_{i=1}^n B_i\) and \(\bigcup_{i=1}^n\overline B_i\subset\Omega\).
Given any \(1\leq i<j\leq n\), we can take a continuous curve \(\gamma_{ij}\)
in \(\Omega\) connecting a point of \(B_i\) to a point of \(B_j\).
Choose \(\delta>0\) so small that
\(\overline B(\gamma_{ij},\delta)\subset\Omega\) for all \(i<j\). Therefore, the set
\(U\coloneqq\bigcup_{i=1}^n B_i\cup\bigcup_{i<j}B(\gamma_{ij},\delta)\)
does the job, as it is a domain containing \(K\) that is compactly
contained in \(\Omega\).
\end{rem}
\begin{lemma}\label{lem:how_to_make_loops_disjoint}
Let \(\sigma\subset\R^2\) be a Jordan loop with \(\ell(\sigma)<+\infty\).
Call \(U\) the bounded connected component of \(\R^2\setminus\sigma\).
Fix a point \(p\in\sigma\) and a compact set \(K\subset U\). Then for
every \(\varepsilon>0\) there exists a Jordan loop \(\sigma'\) in \(\R^2\)
such that \(p\notin\sigma'\), \(\ell(\sigma')\leq\ell(\sigma)+\varepsilon\)
and \(K\subset U'\subset U\), where \(U'\) stands for the bounded connected
component of \(\R^2\setminus\sigma'\).
\end{lemma}
\begin{proof}
First of all, pick a radius \(r\in\big(0,\varepsilon/(2\,\pi)\big)\) such
that the sets \(U\cap\partial B(p,r)\), \(\sigma\setminus\overline B(p,r)\)
and \(\partial B(p,r)\setminus\overline U\) are non-empty. As observed in
Remark \ref{rmk:conn_inside}, we can choose an open, connected set \(V\subset\R^2\)
such that \(K\subset V\subset\overline V\subset U\).
Possibly shrinking \(r\), we can further assume that
\(\overline B(p,r)\cap V=\emptyset\). Fix some points
\(x\in U\cap\partial B(p,r)\) and \(y\in\sigma\setminus\overline B(p,r)\).
We denote by \(\alpha\) the arc in \(\partial B(p,r)\) containing
\(x\), whose interior lies in \(U\) and whose extreme points \(a,b\)
belong to \(\partial B(p,r)\). Notice that \(a\neq b\) as
\(\partial B(p,r)\setminus\overline U\neq\emptyset\). Call
\(\gamma,\gamma'\) the two arcs in \(\sigma\) joining \(a\)
to \(b\); say that \(\gamma\) is the one passing through \(y\).
By construction \(p\in\gamma'\), thus the Jordan loop
\(\sigma'\) obtained by concatenating \(\gamma\) and \(\alpha\) does
not contain \(p\). Given that \(\alpha\setminus\{a,b\}\subset U\),
we see that the bounded connected component \(U'\) of the complement
of \(\sigma'\) is contained in \(U\). Being \(\overline B(p,r)\) and
\(V\) disjoint, we also have \(K\subset U'\). Finally,
we conclude by pointing out that \(\ell(\sigma')=\ell(\gamma)+\ell(\alpha)
\leq\ell(\sigma)+2\pi r<\ell(\sigma)+\varepsilon\).
\end{proof}
\begin{proposition}[Approximation by analytic loops]\label{prop:approx_smooth_loops}
Let \(\sigma\) be a Jordan loop in \(\R^2\) such that \(\ell(\sigma)<+\infty\).
Call \(U\) the bounded connected component of \(\R^2\setminus\sigma\).
Fix \(\varepsilon,\delta>0\) and any compact subset \(K\) of \(U\).
Then there exists an analytic Jordan loop \(\sigma'\) in \(\R^2\)
such that \(\ell(\sigma')\leq\ell(\sigma)+\varepsilon\) and
\(K\subset U'\subset B(U,\delta)\), where \(U'\) stands for the bounded
connected component of the complement of \(\sigma'\).
\end{proposition}
\begin{proof}
The argument is inspired by the proof of \cite[Theorem 3.42]{AmbrosioFuscoPallara}.

Let \(\rho\) be an analytic convolution kernel in \(\R^2\) and denote by
\(f_\varepsilon\coloneqq\chi_U*\rho_\varepsilon\) the mollified
function for all \(\varepsilon>0\). It is well-known
(cf.\ the discussion preceding \cite[Proposition 3.7]{AmbrosioFuscoPallara}) that
\begin{equation}\label{eq:approx_smooth_loops_aux}
\ell(\sigma)=\H^1(\partial U)=
\lim_{\varepsilon\searrow 0}\int_{\R^2}\big|\nabla f_\varepsilon(x)\big|\,{\rm d}x.
\end{equation}
Fix any open, connected set \(V\subset\R^2\) with \(K\subset V\) and
\(\overline V\subset U\), whose existence is proven in Remark \ref{rmk:conn_inside}.
Set \(\delta'\coloneqq\dist(\overline V,\R^2\setminus U)>0\) and choose a
sequence \((\varepsilon_n)_n\subset\big(0,\min\{\delta,\delta'\}\big)\)
converging to \(0\). The superlevel sets \(S^n_t\) are defined as
\(S^n_t\coloneqq\big\{x\in\R^2\,\big|\,f_{\varepsilon_n}(x)>t\}\) for every
\(n\in\N\) and \(t\in(0,1)\). By combining the Morse-Sard theorem with the analytic
implicit function theorem we know that there exists a negligible set \(N\subset(0,1)\)
such that \(S^n_t\) is an analytic domain for
all \(n\in\N\) and \(t\in(0,1)\setminus N\). Notice that
\[
\ell(\sigma)\overset{\eqref{eq:approx_smooth_loops_aux}}=
\lim_{n\to\infty}\int_{\R^2}\big|\nabla f_{\varepsilon_n}(x)\big|\,{\rm d}x
=\lim_{n\to\infty}\int_0^1\H^1(\partial S^n_t)\,{\rm d}t
\geq\int_0^1\varliminf_{n\to\infty}\H^1(\partial S^n_t)\,{\rm d}t
\]
by coarea formula and Fatou lemma. Then there exists \(t\in(0,1)\setminus N\)
for which it holds that \(\varliminf_n\H^1(\partial S^n_t)\leq\ell(\sigma)\).
Let us choose \(n\in\N\) so that \(\H^1(\partial S^n_t)\leq\ell(\sigma)+\varepsilon\).
Given that \(f_{\varepsilon_n}\equiv 0\) on \(\R^2\setminus B(U,\delta)\)
and \(f_{\varepsilon_n}\equiv 1\) on \(V\), one has that
\(V\subset S^n_t\subset B(U,\delta)\). Being \(V\) a connected set, it is entirely
contained in one connected component \(U'\) of \(S^n_t\). Calling \(\sigma'\)
the boundary of \(U'\), we see that \(\sigma'\) is an analytic Jordan loop and
\(\ell(\sigma')\leq\H^1(\partial S^n_t)\leq\ell(\sigma)+\varepsilon\).
Since \(K\subset V\subset U'\subset S^n_t\subset B(U,\delta)\), the statement
is finally achieved.
\end{proof}
\begin{rem}\label{rmk:intersection_1pt}
Let \(\Omega\subset\R^2\) be a bounded domain with \(\H^1(\partial\Omega)<+\infty\).
Let \(U,U'\) be two distinct connected components of \(\R^2\setminus\overline\Omega\).
Call \(\sigma\) (resp.\ \(\sigma'\)) the boundary of \(U\) (resp.\ \(U'\)), which
is a Jordan loop by Theorem \ref{thm:H1_finite_Jordan_domain}. Then it can be
readily checked that \(\sigma\cap\sigma'\) contains at most one point.
\end{rem}
\begin{proposition}\label{prop:boundary_smooth}
Let \(\Omega\subset\R^2\) be a bounded domain with \(\H^1(\partial\Omega)<+\infty\).
Let \(K\subset\R^2\) be a compact set such that \(K\subset\Omega\). Let
\(U\subset\R^2\) be an open set satisfying \(\overline\Omega\subset U\).
Then for every \(\varepsilon>0\) there exists a bounded domain \(\Omega'\subset\R^2\)
with the following properties:
\begin{itemize}
\item[\(\rm i)\)] \(K\subset\Omega'\subset\overline{\Omega'}\subset U\),
\item[\(\rm ii)\)] \(\H^1(\partial\Omega')<\H^1(\partial\Omega)+\varepsilon\),
\item[\(\rm iii)\)] \(\partial\Omega'\) is the union of finitely many pairwise disjoint
analytic Jordan loops.
\end{itemize}
\end{proposition}
\begin{proof}
We denote by \(\{H_i\}_{i\in I}\) the bounded connected components of
\(\R^2\setminus\overline\Omega\), while \(W\) stands for the unbounded
one. We set \(\sigma_i\coloneqq\partial H_i\) for all \(i\in I\) and
\(\sigma\coloneqq\partial W\), which are Jordan loops by Theorem
\ref{thm:H1_finite_Jordan_domain}. Since \(\{H_i\}_{i\in I}\) is an open
covering of the compact set \(\R^2\setminus(U\cup W)\), we can select a
finite subfamily \(F\) of \(I\) such that
\(\R^2\setminus(U\cup W)\subset\bigcup_{i\in F}H_i\). Equivalently, we
have that \((\R^2\setminus W)\setminus\bigcup_{i\in F}H_i\subset U\).
As observed in Remark \ref{rmk:intersection_1pt}, any two of the curves
\(\sigma\) and \(\sigma_i\) can intersect in at most one point, thus
in particular \(\ell(\sigma)+\sum_{i\in F}\ell(\sigma_i)\leq\H^1(\partial\Omega)\).
By repeatedly applying Lemma \ref{lem:how_to_make_loops_disjoint}, we can replace
the curves \(\sigma,\sigma_i\) with some pairwise disjoint Jordan loops
\(\sigma',\sigma'_i\) satisfying the following conditions:
\begin{itemize}
\item[\(\rm a)\)] \(\ell(\sigma')<\ell(\sigma)+\varepsilon/(\sharp F+1)\) and
\(\ell(\sigma'_i)<\ell(\sigma_i)+\varepsilon/(\sharp F+1)\) for all \(i\in F\).
\item[\(\rm b)\)] \((\R^2\setminus U)\cap H_i\) lies in the bounded connected
component \(H'_i\) of \(\R^2\setminus\sigma'_i\) for all \(i\in F\), while
\((\R^2\setminus U)\cap W\) lies in the unbounded connected
component \(W'\) of \(\R^2\setminus\sigma'\).
\item[\(\rm c)\)] \(K\cap\overline{H'_i}=\emptyset\) for all \(i\in F\) and
\(K\cap\overline{W'}=\emptyset\).
\end{itemize}
Thanks to Proposition \ref{prop:approx_smooth_loops}, we can even assume that
the Jordan loops \(\sigma'\) and \(\sigma'_i\) are analytic.
Let us set
\(\Omega'\coloneqq(\R^2\setminus\overline{W'})\setminus\bigcup_{i\in F}\overline{H'_i}\).
The fact that \(\Omega'\) is a bounded domain and item iii) follow from the
very construction of \(\Omega'\). Moreover, i) is granted by b) and c).
Finally, we deduce from a) that
\[
\H^1(\partial\Omega')=\ell(\sigma')+\sum_{i\in F}\ell(\sigma'_i)<
\ell(\sigma)+\sum_{i\in F}\ell(\sigma_i)+\varepsilon
\leq\H^1(\partial\Omega)+\varepsilon,
\]
thus proving item ii). This concludes the proof of the statement.
\end{proof}
\section{Painlev\'{e} length estimates}\label{sec:Painleve}
Let us first recall the definition of Painlev\'e length.

\begin{definition}
The Painlev\'e length of a compact set $K \subset \R^2$, denoted
$\kappa(K)$, is the infimum of numbers $\ell$ with the following property: for every
open set $U$ containing $K$ there exists an open set $V$ such that $K \subset V \subset U$
and $\partial V$ is a finite union of disjoint analytic Jordan curves of total length at
most $\ell$.
\end{definition}

In \cite[p.\ 25]{Dudziak} the following estimate was stated.

\begin{proposition}\label{prop:sharppainleve}
For every compact set $K \subset \R^2$ the following inequality holds:
\begin{equation}\label{eq:sharppainleve}
\kappa(K) \le \pi\,\H^1(K).
\end{equation}
\end{proposition}

In \cite{Dudziak} it was also noted that the estimate \eqref{eq:sharppainleve}
is the best possible one for general compact sets, when the Hausdorff measure
$\H^1$ is replaced by the (possibly smaller) Hausdorff content $\H_\infty^1$.
In what  follows we prove that the
Painlev\'e length estimate \eqref{eq:sharppainleve} can be improved for connected sets.
\begin{theorem}[Painlev\'{e} estimate for connected sets]\label{thm:Painleve_conn}
For every compact, connected set $K \subset \R^2$ the following inequality holds:
\begin{equation}\label{eq:sharppainleve_connected}
\kappa(K) \le 2\,\H^1(K).
\end{equation}
\end{theorem}
\begin{proof}
The case \(\H^1(K)=+\infty\) is trivial, thus let us suppose that \(\H^1(K)<+\infty\).
Fix an open neighbourhood \(U\subset\R^2\) of \(K\) and any \(\varepsilon>0\). 
We aim to prove that there exists an open set \(V\subset U\)
containing \(K\), whose boundary is a disjoint union of finitely many analytic Jordan loops,
such that
\begin{equation}\label{eq:Painleve_connected}
\H^1(\partial V)<2\,\H^1(K)+\varepsilon.
\end{equation}
Fix any positive radius \(r<\dist(K,\R^2\setminus U)\). For any \(x\in K\)
we can choose \(r_x>0\) such that \(r<r_x<\dist(K,\R^2\setminus U)\) and
\(\H^0\big(K\cap\partial B_{r_x}(x)\big)<+\infty\). 
We claim that:
\begin{equation}\label{eq:Painleve_connected_claim1}\begin{split}
&\text{There exist }x_0,\ldots,x_n\in K\text{ and compact connected sets }
x_i\in K_i\subset\bar B_{r_{x_i}}(x_i)\\
&\text{such that }K=K_0\cup\ldots\cup K_n\text{ and }
\H^1(K_i\cap K_j)=0\text{ whenever }0\leq i<j\leq n.
\end{split}\end{equation}
In order to show the validity of the claim \eqref{eq:Painleve_connected_claim1},
we need the following property that can be readily obtained as a consequence
of Lemmata \ref{lem:diam_vs_H1} and \ref{lem:touch_boundary}:\\
{\color{blue}\textsc{Fact.}} If \(E\) is a compact connected subset of \(K\)
and \(x\in E\) satisfies \(E\setminus\bar B_{r_x}(x)\neq\emptyset\), then the
connected component \(F\) of \(E\cap\bar B_{r_x}(x)\) containing \(x\) has
the following properties:
\begin{itemize}
\item[\(\rm i)\)] \(\H^1(F)\geq r_x>r\).
\item[\(\rm ii)\)] \(G\coloneqq E\setminus\big(F\cap B_{r_x}(x)\big)\)
is compact and has finitely many connected components (which accordingly
are compact as well). More precisely, the number of connected components
of \(G\) cannot exceed \(\H^0\big(K\cap\partial B_{r_x}(x)\big)\).
\item[\(\rm iii)\)] \(\H^0(F\cap G)<+\infty\).
\end{itemize}
We now recursively apply the above \textsc{Fact}. First of all,
fix any \(x_0\in K\). If \(K\subset\bar B_{r_{x_0}}(x_0)\)
then we define \(K_0\coloneqq K\). Otherwise, we call \(K_0\) the
connected component of \(K\cap\bar B_{r_{x_0}}(x_0)\) containing \(x_0\).
In the latter case, let us denote by \(E_1,\ldots,E_k\) the connected components
of the set \(K\setminus\big(K_0\cap B_{r_{x_0}}(x_0)\big)\). Given any \(i=1,\ldots,k\),
we pick a point \(x_i\in E_i\) and proceed as before:
if \(E_i\subset\bar B_{r_{x_i}}(x_i)\) then we set \(K_i\coloneqq E_i\);
otherwise, we call \(K_i\) the connected component of \(E_i\cap\bar B_{r_{x_i}}(x_i)\)
containing \(x_i\) and \(E_{i,1},\ldots,E_{i,k_i}\) the connected components
of the set \(E_i\setminus\big(K_i\cap B_{r_{x_i}}(x_i)\big)\).
We can repeat the same argument on each \(E_{i,j}\), and so on.
This iterated procedure must stop after finitely many passages thanks to item i)
of \textsc{Fact} (recall that \(\H^1(K)<+\infty\) and that the intersection
\(K_i\cap K_j\) has null \(\H^1\)-measure if \(i\neq j\)).
Therefore the previous argument provides us with a family \(K_0,\ldots,K_n\)
as in claim \eqref{eq:Painleve_connected_claim1}.

Let us now denote by \(\tilde C_i\) the convex hull of \(K_i\) for every \(i=0,\ldots,n\).
We know from Proposition \ref{prop:boundary_convex_hull} that
\(\H^1(\partial\tilde C_i)\leq 2\,\H^1(K_i)\). Moreover, recall that the
open \(\delta\)-neighbourhood \(\tilde C_i^\delta\) of \(\tilde C_i\) is convex for
all \(\delta>0\) and satisfies
\(\H^1(\partial\tilde C_i^\delta)\to\H^1(\partial\tilde C_i)\) as \(\delta\to 0\).
Furthermore, given any \(i=0,\ldots,n\) and any Borel set \(F\subset\R^2\)
with \(\H^1(F)<+\infty\), it holds that
\[\H^0(\partial\tilde C_i^\delta\cap F)<+\infty\quad\text{ for a.e.\ }\delta>0.\]
By \eqref{eq:Painleve_connected_claim1} we have that
each \(K_i\) is contained in the convex set \(\bar B_{r_{x_i}}(x_i)\),
whence accordingly the inclusions \(\tilde C_i\subset\bar B_{r_{x_i}}(x_i)\subset U\) hold
for every \(i=0,\ldots,n\). Hence we can recursively choose \(\delta_0,\ldots,\delta_n>0\)
so that (calling \(C_i\coloneqq\tilde C_i^{\delta_i}\)) the following properties
are verified:
\begin{itemize}
\item[\(\rm a)\)] \(C_i\subset U\) for every \(i=0,\ldots,n\),
\item[\(\rm b)\)] \(\H^1(\partial C_i)\leq 2\,\H^1(K_i)+\varepsilon/(n+1)\)
for every \(i=0,\ldots,n\).
\end{itemize}
Let us now define \(V\coloneqq C_0\cup\ldots\cup C_n\).
Then \(V\) is an open set such that \(K\subset V\subset U\)
and \(\partial V\subset\partial C_0\cup\ldots\cup\partial C_n\),
thus it holds that
\[\H^1(\partial V)\leq\sum_{i=0}^n \H^1(\partial C_i)
\overset{\rm b)}\leq
2\sum_{i=0}^n\H^1(K_i)+\varepsilon=2\,\H^1(K)+\varepsilon,\]
proving \eqref{eq:Painleve_connected}. 
The fact that the boundary of \(V\) can be supposed to be made of 
finitely many disjoint analytic Jordan loops is due to Proposition \ref{prop:boundary_smooth}.
This completes the proof.
\end{proof}
%
%
The estimate \eqref{eq:sharppainleve_connected} is easily seen to be sharp simply by taking $K$ to be a line-segment. What is less trivial, is that also the estimate \eqref{eq:sharppainleve} is sharp for general compact sets. This is shown in the next example.

\begin{example}\label{ex:example_sharp_painleve}
We define a compact fractal set $K\subset \R^2$ with $\kappa(K) = \pi\H^1(K)$ using an iteration procedure. We start with $K_1 = \overline{B}((0,0),1)$ and continue by contracting and copying $K_1$ as follows. For each integers $k$ and $j$ with $k \ge 2$ and $1 \le j \le 2^k$ we define, using complex notation, a contractive similitude
\[
f_{k,j}(x) := 2^{-k}x + (1-2^{-k})e^{j2^{1-k}\pi i}.
\]
Using these functions we set
\[
K_k := \bigcup_{(j_2,\dots,j_k)}f_{2,j_2}\circ f_{3,j_3} \circ \cdots \circ f_{k,j_k}(K_1),
\]
where the union runs over all $(k-1)$-tuples of indices with $1 \le j_i \le 2^k$, $i=2,\dots,k$. We call the balls in this union the construction balls of level $k$.
Notice that $K_{k+1}\subset K_k$ for every $k\in \N$. Finally, we set
\[
K := \bigcap_{k=1}^\infty K_k.
\]
See Figure \ref{fig:example} for an illustration of the construction.
\begin{figure}
    \centering
    \includegraphics[width=0.6\columnwidth]{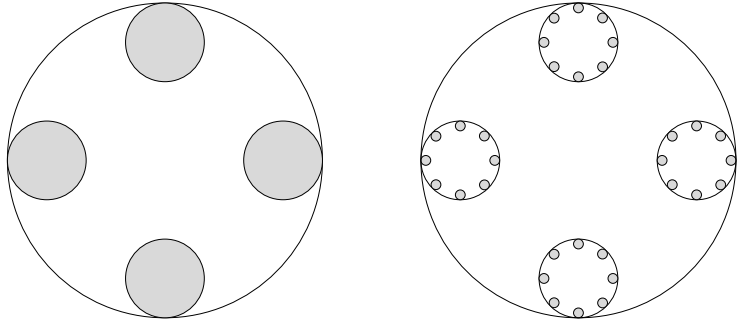}
    \caption{Here are approximations of the set $K$ defined in Example \ref{ex:example_sharp_painleve}.
    On the left is $K_2$ and on the right one step further refinement, the set $K_3$.}
    \label{fig:example}
\end{figure}
We claim that $\H^1(K) = 2$ and $\kappa(K) = 2\pi$. Taking into account \eqref{eq:sharppainleve}, it suffices to show that $\H^1(K) \le 2$ and $\kappa(K) \ge 2\pi$.

The inequality $\H^1(K) \le 2$ follows directly by using the construction balls of level $k$ as the cover for $K$ in the definition of the Hausdorff measure and by letting $k \to \infty$. Thus, it only remains to show that $\kappa(K) \ge 2\pi$.

Let $\varepsilon > 0$. We show that
\begin{equation}\label{eq:examplepainleve}
\kappa(K) \ge (1-\varepsilon)2\pi.
\end{equation}
By letting $U \subset \R^2$ be an open set containing $K$ such that each connected component of $U$ contains only one construction ball of level $k$, we may restrict ourselves to estimating $\kappa(K \cap B)$ for a construction ball $B$ of level $k$ with $k$ arbitrarily large. Let $V \subset U$ be open such that $K\cap B \subset V$. It suffices to show that for each connected component $V'$ of $V$
we have $\H^1(\partial V') \ge (1-\varepsilon)\pi\H^1(V'\cap K)$. Since $\H^1(\partial V') \ge \H^1(\partial W)$ for $W=\conv(V'\cap K)$, it is enough to show that $\H^1(\partial W) \ge (1-\varepsilon)\pi\H^1(V'\cap K)$.

Let $k_0$ be the smallest integer so that $W$ intersects at least $2$ of the level $k_0$ construction balls. By our assumption on $U$ we have that $k_0 > k$. The set $W$ is then contained in a level $k_0-1$ ball $B(x,r)$. We separate the rest of the proof into two cases:
\begin{itemize}
\item[\(\rm i)\)] $W$ intersects exactly $2$ level $k_0$ construction balls.
\item[\(\rm ii)\)] $W$ intersects at least $3$ level $k_0$ construction balls.
\end{itemize}

Let us first consider the case \(\rm i)\). Since the distance between two level $k_0$ construction balls is at least
\[
(1-2^{1-k_0})\sin(2^{-k_0}\pi) r \ge (1-2^{2-k_0})2^{-k_0}\pi r, 
\]
we may assume that $\H^1(V'\cap K) \ge 2^{-k_0}r$. Then, one of the construction balls contains a point of $\partial W$ that has distance at least $\frac32(1-2^{2-k_0})2^{-k_0}\pi r$ to the other construction ball. Thus, we may assume that $\H^1(V'\cap K) \ge \frac322^{-k_0}r$. But then, there exist two points in $\partial W$ with distance at least $2(1-2^{2-k_0})2^{-k_0}\pi r$, which then yields
\[
\H^1(\partial W) \ge (1-2^{2-k_0}) 2\pi r 2^{2-k_0} \ge (1-2^{2-k_0})\pi\H^1(V' \cap K).
\]

Let us then consider the case \(\rm ii)\). For each construction level $k_0$ ball $B_i$ intersecting $W$ there exists a point $x_i \in \partial W \cap B_i$, since none of the balls $B_i$ is in the convex hull of the other balls. Let us then estimate $\H^1(\partial W)$ using the angle around  the center $x$. If $x_i$ and $x_{j}$ are contained in adjacent construction balls, the boundary of $W$ from $x_i$ to $x_{j}$ has length at least
\[
(1-2^{1-k_0})\sin(\alpha_{i,j}) r \ge (1-2^{2-k_0})\alpha_{i,j} r,
\]
where $\alpha_{i,j} := \measuredangle(x_i,x,x_j)$. See Figure \ref{fig:example_proof} for an illustration for the estimate. If $x_i$ and $x_{j}$ are not contained in adjacent construction balls, the length of the boundary of $W$ from $x_i$ to $x_{j}$ is at least $2^{1-k_0}\pi$. All in all, denoting by $N$ the total number of the construction balls $B_i$ intersecting $W$, we have
\[
\H^1(\partial W) \ge (1-2^{2-k_0}) 2\pi r N 2^{-k_0} \ge (1-2^{2-k_0})\pi\H^1(V' \cap K).
\]
\begin{figure}
    \centering
    \includegraphics[width=0.45\columnwidth]{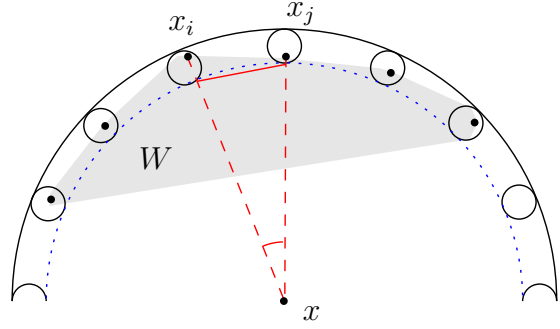}
    \caption{In the case \(\rm ii)\), the length of $\partial W$ is estimated by summing up lengths of projections (red line) of parts connecting points in consecutive balls (here between $x_i$ and $x_j$).}
    \label{fig:example_proof}
\end{figure}
\end{example}
\section{Proof of the main result}\label{sec:proof}
This section is entirely devoted to the proof of Theorem \ref{thm:main}:
\medskip

\noindent{\color{blue}\textsc{Step 1.}}
Let us denote by \(\{E_i\}_{i=1}^N\) the connected components of
\(\partial\Omega\) with positive length, where \(N\in\N\cup\{\infty\}\).
We can clearly suppose without loss of generality that \(N=\infty\).
In the case in which \(\Omega\) is bounded, we also assume that \(E_1\)
is the element containing the boundary of the unbounded connected
component of \(\R^2\setminus\overline\Omega\), which is connected as it
is a Jordan loop by Theorem \ref{thm:H1_finite_Jordan_domain}.
Set \(C\coloneqq\partial\Omega\setminus\bigcup_{i=1}^\infty E_i\).
Notice that \(C\) can be a Cantor-type set, thus in particular it can have
positive \(\H^1\)-measure. Lemma \ref{lem:diam_vs_H1} grants that
\begin{equation}\label{eq:main_thm_1}
\sum_{i=1}^\infty \diam(E_i)\leq\sum_{i=1}^\infty\H^1(E_i)
\leq\H^1(\partial\Omega)<+\infty.
\end{equation}
Consequently, we can relabel the sets \(\{E_i\}_{i\geq 2}\)
so that \(\diam(E_i)\geq\diam(E_j)\) if \(2\leq i\leq j\).

Let us fix \(\varepsilon>0\). For each \(i\in\N\), we select a point \(z_i\in E_i\).
Observe that
\[
\sum_{i=1}^\infty\H^1_\infty\big(B(z_i,4\,\diam(E_i))\big)
\leq 8\sum_{i=1}^\infty\diam(E_i)\leq 8\sum_{i=1}^\infty\H^1(E_i)
\leq 8\,\H^1(\partial\Omega)<+\infty.
\]
By using the Borel-Cantelli lemma we deduce that
\(\H^1_\infty\big(\bigcap_{k=1}^\infty\bigcup_{i=k}^\infty B(z_i,4\,\diam(E_i))\big)=0\).
Since \(\H^1\ll\H^1_\infty\) and the measure \(\H^1|_{\partial\Omega}\) is continuous
from above, we see that
\[
\lim_{k\to\infty}\H^1\Big(\partial\Omega\cap
\bigcup\nolimits_{i=k}^\infty B\big(z_i,4\,\diam(E_i)\big)\Big)
=\H^1\Big(\partial\Omega\cap\bigcap\nolimits_{k=1}^\infty\bigcup\nolimits_{i=k}^\infty
B\big(z_i,4\,\diam(E_i)\big)\Big)=0.
\]
Therefore, there exists \(k\in\N\) such that
\begin{equation}\label{eq:kcond}\begin{split}
\H^1\Big(\partial\Omega\cap\bigcup\nolimits_{i=k}^\infty B\big(z_i,4\,\diam(E_i)\big)\Big)
&<\frac{2\,\varepsilon}{5\,\pi},\\
\sum_{i=k}^\infty\diam(E_i)&<\frac{\varepsilon}{10\,\pi}.
\end{split}\end{equation}
{\color{blue}\textsc{Step 2.}}
Choose any continuous curve \(\gamma_1\subset\Omega\) joining \(x\) to \(y\).
Theorem \ref{thm:Painleve_conn} grants that for any \(i=1,\ldots,k-1\) we can
choose an open neighbourhood \(V_i\) of \(E_i\)
in such a way that
\[\begin{split}
\overline V_i\cap\overline V_j=\emptyset&\quad\text{ for every }1\leq i<j\leq k-1,\\
\overline V_i\cap\gamma_1=\emptyset&\quad\text{ for every }i=1,\ldots,k-1,\\
\H^1(\partial V_i)\leq 2\,\H^1(E_i)+\frac{2\,\varepsilon}{5\,(k-1)}&
\quad\text{ for every }i=1,\ldots,k-1.
\end{split}\]
We can also assume that the boundary of each set \(V_i\) consists of finitely
many pairwise disjoint Jordan loops. Notice that \(x,y\) lie in the same connected
component of \(\R^2\setminus\bigcup_{i=1}^{k-1}V_i\), thanks to the fact
that the curve \(\gamma_1\) does not intersect \(\bigcup_{i=1}^{k-1}V_i\).
We distinguish two cases:
\begin{itemize}
\item \(\Omega\) is bounded. Let us call \(\Omega'\) the (bounded) connected
component of \(\R^2\setminus\overline V_1\) that contains \(\gamma_1\)
(thus also \(x,y\)). The boundary of \(\Omega'\) is a Jordan loop
\(\sigma\colon[0,1]\to\R^2\) with
\(\ell(\sigma)\leq\H^1(\partial V_1)\leq 2\,\H^1(E_1)+\frac{2\,\varepsilon}{5(k-1)}\)
(cf.\ Theorem \ref{thm:H1_finite_Jordan_domain}). Possibly reparametrizing
\(\sigma\), we can suppose to have \(0<t_1<t_2<t_3<t_4<t_5<1\) such that
\[
x,y\in(\sigma_0,\sigma_{t_3})\subset\Omega',\quad
x\in(\sigma_{t_1},\sigma_{t_5})\subset\Omega',\quad
y\in(\sigma_{t_2},\sigma_{t_4})\subset\Omega'
\]
and the segments \([\sigma_{t_1},\sigma_{t_5}],[\sigma_{t_2},\sigma_{t_4}]\)
are perpendicular to \([x,y]\). It readily follows from Lemma \ref{lem:diam_vs_H1}
that \(|x-\sigma_{t_1}|\leq\ell(\sigma|_{[0,t_1]})\), 
\(|y-\sigma_{t_2}|\leq\ell(\sigma|_{[t_2,t_3]})\),
\(|y-\sigma_{t_4}|\leq\ell(\sigma|_{[t_3,t_4]})\) and 
\(|x-\sigma_{t_5}|\leq\ell(\sigma|_{[t_5,0]})\).
Calling \(*\) the concatenation of curves, we thus see that
\[
\ell\big([x,\sigma_{t_1}]*\sigma|_{[t_1,t_2]}*[\sigma_{t_2},y]\big)
+\ell\big([x,\sigma_{t_5}]*\sigma|_{[t_5,t_4]}*[\sigma_{t_4},y]\big)\leq\ell(\sigma).
\]
Then we define \(s\colon[0,1]\to\R^2\) as the shortest curve between
\([x,\sigma_{t_1}]*\sigma|_{[t_1,t_2]}*[\sigma_{t_2},y]\) and
\([x,\sigma_{t_5}]*\sigma|_{[t_5,t_4]}*[\sigma_{t_4},y]\), so
that \(\ell(s)\leq\H^1(E_1)+\frac{\varepsilon}{5(k-1)}\). See Figure \ref{fig:initial} for the curve $s$.
\begin{figure}
    \centering
    \includegraphics[width=0.6\columnwidth]{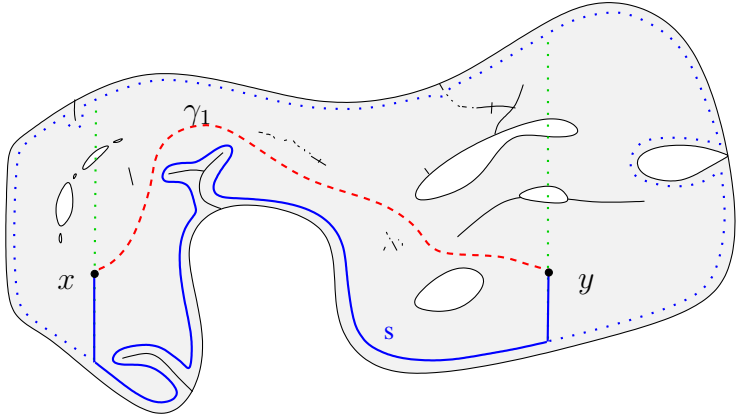}
    \caption{The first approximation $s$ of the curve is obtained in the bounded case by going near the outer boundary. In the unbounded case, the initial curve is just the line-segment connecting the points.}
    \label{fig:initial}
\end{figure}
\item \(\Omega\) is unbounded. Then we define \(s\colon[0,1]\to\R^2\)
as \(s_t\coloneqq x+t(y-x)\) for all \(t\in[0,1]\).
\end{itemize}
For the sake of simplicity, let us define the quantity \(q>0\) as
\begin{equation}\label{eq:def_q}
q\coloneqq\left\{\begin{array}{ll}
0\\
|x-y|
\end{array}\quad\begin{array}{ll}
\text{ if }\Omega\text{ is bounded,}\\
\text{ if }\Omega\text{ is unbounded.}
\end{array}\right.
\end{equation}
\begin{figure}
    \centering
    \includegraphics[width=0.6\columnwidth]{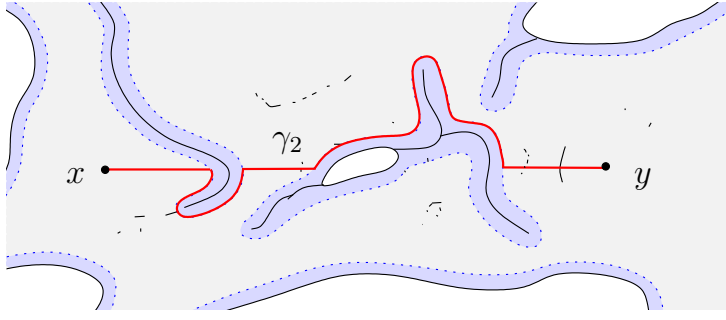}
    \caption{In the second approximation, the curve is constructed so that it avoids a finite number of the largest boundary components. Next, it is slightly perturbed so that it does not intersect the (remaining) boundary points in a positive $\mathcal H^1$-measure set.}
    \label{fig:avoid_large}
\end{figure}

We proceed in a recursive way: choose that \(i_1\in\{1,\ldots,k-1\}\)
such that \(V_{i_1}\) is the first element of \(\{V_i\}_{i=1}^{k-1}\)
that is encountered by the curve \(s\) (note that \(i_1\geq 2\) if \(\Omega\) is bounded).
Put \(a_1\coloneqq\min\big\{t\in(0,1)\,\big|\,s_t\in\partial V_{i_1}\big\}\).
The connected component of \(\partial V_{i_1}\) containing \(s_{a_1}\) is the image
of a Jordan loop \(\sigma_1\). Now let us call
\(b_1\coloneqq\max\big\{t\in(a_1,1)\,\big|\,s_t\in\sigma_1\big\}\).
Observe that \(s|_{(b_1,1)}\cap\partial V_{i_1}=\emptyset\).
We can write the image of \(\sigma_1\) as the union of two injective
curves \(\alpha^1,\tilde\alpha^1\) joining \(s_{a_1}\) to \(s_{b_1}\).
Given that the length of \(\sigma_1\) does not exceed \(\H^1(\partial V_{i_1})\),
which in turn is smaller than \(2\,\H^1(E_{i_1})+\frac{2\,\varepsilon}{5(k-1)}\),
we can assume without loss of generality that the length of \(\alpha^1\)
is smaller than \(\H^1(E_{i_1})+\frac{\varepsilon}{5(k-1)}\).

We can now argue in the same way starting from \(s_{b_1}\). Take
\(i_2\in\{1,\ldots,k-1\}\setminus\{i_1\}\) such that the first of
the sets \(V_i\) that we meet while going from \(s_{b_1}\) to \(y\) is
\(V_{i_2}\) (again, \(i_2\neq 1\) if \(\Omega\) is bounded).
We denote by \(a_2\) the smallest \(t\in(b_1,1)\) for which
\(s_t\in\partial V_{i_2}\); the connected component of \(\partial V_{i_2}\) containing
\(s_{a_2}\) is the image of a Jordan loop \(\sigma_2\), and \(b_2\)
stands for the biggest \(t\in(a_2,1)\) such that \(s_t\in\sigma_2\).
Then we can find a curve \(\alpha^2\) in \(\sigma_2\) joining \(s_{a_2}\)
to \(s_{b_2}\), which is shorther than \(\H^1(E_{i_2})+\frac{\varepsilon}{5(k-1)}\).
By repeating this procedure finitely many times (see Figure \ref{fig:avoid_large}), we obtain a curve \(\gamma_2\)
of the form
\[
\gamma_2\coloneqq s|_{[0,a_1]}*\alpha^1*s|_{[b_1,a_2]}*\alpha^2*\ldots*
s|_{[b_{\ell-1},a_\ell]}*\alpha^\ell*s|_{[b_\ell,1]}
\]
for some \(\ell\leq k-1\).
Notice that \(\gamma_2\) is contained in \(\R^2\setminus\bigcup_{i=1}^{k-1}E_i\)
and connects \(x\) to \(y\). By combining the previous estimates, we also deduce that
\begin{equation}\label{eq:curve_gamma2}
\ell(\gamma_2)<q+\sum_{i=1}^{k-1}\H^1(E_i)+\frac{\varepsilon}{5}.
\end{equation}
{\color{blue}\textsc{Step 3.}}
In light of \eqref{eq:curve_gamma2}, we can choose some points
\(p_1,\ldots,p_{2h-1}\in\gamma_2\setminus\{x,y\}\) having the following property:
the curve \(\gamma_3\coloneqq[x,p_1]*[p_1,p_2]*\ldots*[p_{2h-2},p_{2h-1}]*[p_{2h-1},y]\)
is contained in \(\R^2\setminus\bigcup_{i=1}^{k-1}E_i\) and satisfies
\[
\ell(\gamma_3)<q+\sum_{i=1}^{k-1}\H^1(E_i)+\frac{\varepsilon}{5}.
\]
Now let us apply Proposition \ref{prop:move_points}: we can find some points
\(q_1,q_3,\ldots,q_{2h-1}\) (sufficiently near to \(p_1,p_3,\ldots,p_{2h-1}\),
respectively) for which the following conditions are verified:
\begin{itemize}
\item The curve
\(\gamma_4\coloneqq[x,q_1]*[q_1,p_2]*\ldots*[p_{2h-2},q_{2h-1}]*[q_{2h-1},y]\)
satisfies
\begin{equation}\label{eq:curve_gamma4}
\ell(\gamma_4)<q+\sum_{i=1}^{k-1}\H^1(E_i)+\frac{\varepsilon}{5},
\end{equation}
\item \(\gamma_4\) is contained in \(\R^2\setminus\bigcup_{i=1}^{k-1}E_i\),
\item the set \(\gamma_4\cap\partial\Omega\) has null \(\H^1\)-measure.
\end{itemize}
By upper continuity of \(\H^1|_{\partial\Omega}\),
we can find \(\delta>0\) such that
\(B(\gamma_4,2\,\delta)\subset\R^2\setminus\bigcup_{i=1}^{k-1}E_i\) and
\begin{equation}\label{eq:nbhd_gamma4}
\H^1\big(\partial\Omega\cap B(\gamma_4,2\,\delta)\big)<\frac{2\,\varepsilon}{5\,\pi}.
\end{equation}
Theorem \ref{thm:Painleve_conn} provides an open neighbourhood
\(U'\subset B(\gamma_4,\delta)\setminus\bigcup_{i=1}^{k-1}E_i\) of \(\gamma_4\) such that
\begin{equation}\label{eq:estimate_U'}
\H^1(\partial U')\leq 2\,\ell(\gamma_4)+\frac{2\,\varepsilon}{5}
\leq 2\,q+2\sum_{i=1}^{k-1}\H^1(E_i)+\frac{4\,\varepsilon}{5},
\end{equation}
where the second inequality stems from \eqref{eq:curve_gamma4}.
Moreover, let us fix any index \(m\geq k\) for which
\(\diam(E_m)<\dist(\gamma_4,\partial U')\). Since \(i\mapsto\diam(E_i)\)
is non-increasing for \(i\geq k\), one has
\begin{equation}\label{eq:diam_Ei_small}
\diam(E_i)<\dist(\gamma_4,\partial U')\quad\text{ for every }i\geq m.
\end{equation}
Let us define
\[
U\coloneqq U'\cup\bigcup_{i=k}^m B\big(z_i,2\,\diam(E_i)\big).
\]
See Figure \ref{fig:final} for an illustration of the set $U$.
\begin{figure}
    \centering
    \includegraphics[width=0.6\columnwidth]{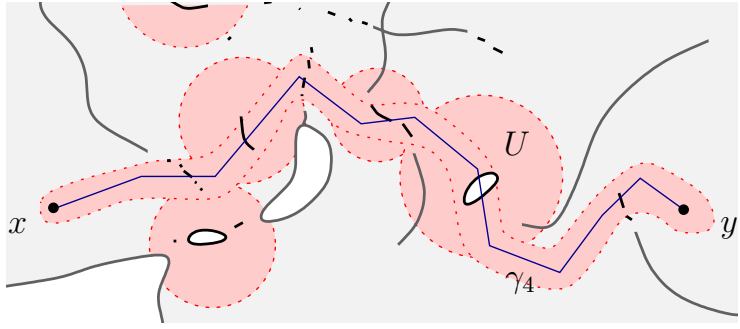}
    \caption{The final curve is found inside a set $U$ that is obtained as the union of a neighbourhood of the curve $\gamma_4$ and suitable collection of balls. The neighbourhood of $\gamma_4$ allows us to avoid the largest pieces of the boundary, which the curve $\gamma_4$ avoided. By taking balls around the remaining boundary parts that have so large diameter that they could block inside the previous neighbourhood, we are guaranteed to be able to connect $x$ to $y$ inside $U$.}
    \label{fig:final}
\end{figure}
Given that \(\overline U\subset B(\gamma_4,2\,\delta)\cup
\bigcup_{i=k}^m B\big(z_i,4\,\diam(E_i)\big)\),
we deduce from the first line of \eqref{eq:kcond} and from \eqref{eq:nbhd_gamma4} that
\begin{equation}\label{eq:bdry_Omega_in_U}
\H^1(\partial\Omega\cap\overline U)<\frac{4\,\varepsilon}{5\,\pi}.
\end{equation}
{\color{blue}\textsc{Step 4.}} We claim that
\begin{equation}\label{eq:main_thm_claim}
x,y\text{ belong to the same connected component of }U\setminus\partial\Omega.
\end{equation}
We argue by contradiction: suppose that \(y\) does not belong to the connected
component \(A\) of \(U\setminus\partial\Omega\) containing \(x\). Call \(B\)
the connected component of \(\R^2\setminus A\) that contains \(y\) (notice that
\(y\) lies in the interior of \(B\)). Call \(F\) the connected component of
\(\partial A\) that is included in \(B\). Hence, Lemma \ref{lem:fact_HH} yields
\(F=\partial B\). Given that \(\gamma_4\) joins \(x\notin B\) to \(y\in B\),
we deduce that \(\gamma_4\cap\partial B\neq\emptyset\). Choose any
\(z\in\gamma_4\cap\partial B\). Observe that
\(\partial B\subset\partial A\subset\partial U\cup\partial\Omega\),
thus the fact that \(\gamma_4\subset U\) gives \(z\in\partial\Omega\).
Call \(E\) the connected component of \(\partial\Omega\) containing \(z\).
Given that \(\gamma_4\subset\R^2\setminus\bigcup_{i=1}^{k-1}E_i\), we have that
either \(E=E_i\) for some \(i\geq k\) or \(E\subset C\). To prove that
\(E\cap\partial U=\emptyset\) we distinguish the following three cases:
\begin{itemize}
\item[\(\rm i)\)] \(E=E_i\) for some \(i=k,\ldots,m\). Then it holds
\(E\subset B\big(z_i,2\,\diam(E_i)\big)\subset U\), whence accordingly
\(E\cap\partial U=\emptyset\).
\item[\(\rm ii)\)] \(E=E_i\) for some \(i>m\). Since
\(\diam(E)<\dist(\gamma_4,\partial U')\) by \eqref{eq:diam_Ei_small}
and \(\gamma_4\cap E\neq\emptyset\), we see that \(E\subset U'\subset U\)
and thus \(E\cap\partial U=\emptyset\).
\item[\(\rm iii)\)] \(E\subset C\). Then \(E\) is a non-empty connected set with
null diameter, namely a singleton, so that clearly \(E\cap\partial U=\emptyset\).
\end{itemize}
Therefore, \(E\) is also a connected component of \(\partial U\cup\partial\Omega\)
and accordingly \(\partial B\subset E\). The set \(\R^2\setminus\partial B\)
cannot be connected, otherwise we would have
\(\mathring B=\R^2\setminus\partial B\) and thus \(B=\R^2\).
Therefore, \(\R^2\setminus\partial B\) has at least two connected components:
one coincides with \(\mathring B\) (thus contains \(y\)), while another one contains
the point \(x\). Then any curve joining \(x\) to \(y\) must intersect
\(\partial B\subset\partial\Omega\), which is in contradiction with the assumption that
\(\Omega\) is connected. Consequently, the claim \eqref{eq:main_thm_claim} is proven.\\
{\color{blue}\textsc{Step 5.}} Thanks to \eqref{eq:main_thm_claim}, we can find
a continuous curve \(\gamma_5\subset U\setminus\partial\Omega\) joining \(x\)
to \(y\). The Painlev\'{e} estimate (for general compact sets),
namely Proposition \ref{prop:sharppainleve}, provides us
with an open neighbourhood \(V\) of \(\partial\Omega\cap\overline U\) such
that \(\overline V\cap\gamma_5=\emptyset\) and
\begin{equation}\label{eq:estimate_bdry_V}
\H^1(\partial V)\leq\pi\,\H^1(\partial\Omega\cap\overline U)<\frac{4\,\varepsilon}{5},
\end{equation}
where the second inequality is a consequence of \eqref{eq:bdry_Omega_in_U}.
Let us denote by \(W'\) the connected component of \(U\setminus\overline V\)
containing \(\gamma_5\). Note that \(\partial W'\subset\partial U'\cup\partial V\cup
\bigcup_{i=k}^m\partial B\big(z_i,2\,\diam(E_i)\big)\).
Therefore, by combining the estimates in \eqref{eq:estimate_U'}, in
\eqref{eq:estimate_bdry_V} and in the second line of \eqref{eq:kcond}, we conclude that
\(\H^1(\partial W')<2\big(q+\sum_{i=1}^{k-1}\H^1(E_i)+\varepsilon\big)\).
Since \(\gamma_5\subset W'\subset\overline{W'}\subset\R^2\setminus\partial\Omega\),
we can apply Proposition \ref{prop:boundary_smooth} to obtain a bounded domain
\(W\subset\R^2\) with \(\gamma_5\subset W\subset\Omega\), whose boundary is the
disjoint union of finitely many smooth Jordan loops and such that
\begin{equation}\label{eq:estimate_W}
\H^1(\partial W)<2\Big(q+\sum\nolimits_{i=1}^{k-1}\H^1(E_i)+\varepsilon\Big).
\end{equation}
We call \(\lambda\) the boundary of the unbounded connected component of
\(\R^2\setminus\overline W\), while by \(\{\lambda_j\}_{j\in J}\) (for some
finite family of indices \(J\)) we denote the boundaries of the bounded connected
components of \(\R^2\setminus\overline W\). Let us also define
\(\Lambda\coloneqq\bigcup_{j\in J}\lambda_j\).\\
{\color{blue}\textsc{Step 6.}}
Call \(L_x\) and \(L_y\) the lines orthogonal to \([x,y]\) that pass
through \(x\) and \(y\), respectively. Take those points \(u_1,u_2,u_3,u_4\in\lambda\)
such that \(x\in[u_1,u_3]\subset L_x\), \(y\in[u_2,u_4]\subset L_y\),
and \((u_1,u_3)\cap\lambda,(u_2,u_4)\cap\lambda=\emptyset\). We can suppose
that \(u_1,u_2\) lie in the same connected component of \(\R^2\setminus\R(y-x)\)
(thus \(u_3,u_4\) are contained in the other one). By \(\wideparen{u_1 u_2}\) we
mean the arc in \(\lambda\) joining \(u_1\) to \(u_2\) that does not contain
any other point \(u_i\), similarly for \(\wideparen{u_3 u_4}\) and so on.
The set \(\R(y-x)\setminus[x,y]\) is the union of two half-lines; both of them
intersect the curve \(\lambda\), say at some points \(u_5\in\wideparen{u_1 u_3}\)
and \(u_6\in\wideparen{u_2 u_4}\). By Lemma \ref{lem:diam_vs_H1} we see that
\begin{equation}\label{eq:estimate_segments}\begin{split}
|x-u_1|\leq\H^1(\wideparen{u_1 u_5}),\quad
&|y-u_2|\leq\H^1(\wideparen{u_2 u_6}),\quad\\
|x-u_3|\leq\H^1(\wideparen{u_3 u_5}),\quad
&|y-u_4|\leq\H^1(\wideparen{u_4 u_6}).
\end{split}\end{equation}
Let us define the curves \(\gamma_6,\gamma_7\) as
\[
\gamma_6\coloneqq[x,u_1]*\wideparen{u_1 u_2}*[u_2,y],\qquad
\gamma_7\coloneqq[x,u_3]*\wideparen{u_3 u_4}*[u_4,y].
\]
Therefore, \eqref{eq:estimate_segments} ensures that
\(\ell(\gamma_6)+\ell(\gamma_7)\leq\ell(\lambda)\), whence (possibly relabeling
\(\gamma_6\) and \(\gamma_7\)) it holds that \(\ell(\gamma_6)\leq\ell(\lambda)/2\).
Finally, it can be readily checked that it is possible to find a curve
\(\gamma\subset\gamma_6\cup\Lambda\) joining \(x\) to \(y\) such that
\(\gamma\cap\partial\Omega=\emptyset\) (thus \(\gamma\subset\Omega\))
and \(\H^1(\gamma\cap\Lambda)\leq\H^1(\Lambda)/2\). Consequently, we deduce that
\(\ell(\gamma)\leq\ell(\gamma_6)+\H^1(\gamma\cap\Lambda)\leq
\big(\ell(\lambda)+\H^1(\Lambda)\big)/2=\H^1(\partial W)/2\).
By recalling the inequality \eqref{eq:estimate_W}, we conclude that
\(\ell(\gamma)\leq q+\sum_{i=1}^{k-1}\H^1(E_i)+\varepsilon\).
In view of \eqref{eq:def_q}, this explicitly means that
\[
\ell(\gamma)\leq\left\{\begin{array}{ll}
\sum_{i=1}^{k-1}\H^1(E_i)+\varepsilon\\
|x-y|+\sum_{i=1}^{k-1}\H^1(E_i)+\varepsilon
\end{array}\quad\begin{array}{ll}
\text{ if }\Omega\text{ is bounded,}\\
\text{ if }\Omega\text{ is unbounded.}
\end{array}\right.
\]
By arbitrariness of \(\varepsilon>0\), this completes the proof of Theorem \ref{thm:main}.
\bibliographystyle{amsplain}
\bibliography{references}

\end{document}